\documentclass[12pt]{article}

\usepackage[dvips]{graphicx}
\usepackage[T1]{fontenc}
\usepackage[utf8]{inputenc}
\usepackage{authblk}

\usepackage[numbers]{natbib}
\usepackage[OT4]{fontenc}
\usepackage[utf8]{inputenc}

\usepackage{amsthm}
\usepackage{amsmath}
\usepackage{bbm}
\usepackage{amssymb}

\newcommand{\Exp}{{\rm I\hspace{-0.8mm}E}}
\newcommand{\Prob}{{\rm I\hspace{-0.8mm}P}}

\newcommand{\iz}{{\rm \rlap Z\kern 2.2pt Z}}

  \def\R{{\mathbb R}}
  
  \def\td{\text{\rm d}}

\newcommand{\ind}{{1\hspace{-1mm}{\rm I}}}

\newtheorem{theorem}{Theorem}

\newtheorem{proposition}{Proposition}
\newtheorem{definition}{Definition}

\newtheorem{remark}{Remark}
\newtheorem{example}{Example}

\title{{\bf  The distribution of the supremum \\
for spectrally asymmetric \\ L\'evy processes}}

\author[1]{Zbigniew Michna\footnote{Corresponding author\\ Email: zbigniew.michna@ue.wroc.pl\\ Tel/fax: +48713680335}}
\author[2]{Zbigniew Palmowski}
\author[3]{Martijn Pistorius}

\affil[1]{Department of Mathematics and Cybernetics

Wroc{\l}aw University of Economics}

\affil[2]{Mathematical Institute

University of Wroc{\l}aw}

\affil[3]{Department of Mathematics

Imperial College London}

\date{}

\begin{document}

\baselineskip=17pt

\maketitle

\begin{abstract}
In this article we derive formulas for the probability 
$\Prob(\sup_{t\leq T} X(t)>u)$ $T>0$ and
$\Prob(\sup_{t<\infty} X(t)>u)$
where $X$ is a spectrally positive L\'evy process with infinite variation. The formulas are generalizations
of the well-known Tak\'acs formulas for stochastic processes with non-negative and interchangeable increments. Moreover, 
we find the joint distribution of $\inf_{t\leq T} Y(t)$ and $Y(T)$ where $Y$ is a spectrally negative L\'evy process.

\vspace{5mm}
\newline
\vspace{2cm}
\end{abstract}

\section{Introduction}
L\'evy processes arise in many areas of probability and play an important role among stochastic processes. The distribution of the supremum of a stochastic process especially of a L\'evy process
appears in many applications in finance, insurance, queueing systems and engineering.
In this article we will consider spectrally asymmetric L\'evy processes,
that is, L\'evy processes with L\'evy measure concentrated on $(0,\infty)$ or $(-\infty,0)$ (these have only positive jumps or only negative jumps, respectively). The problem to determine the distribution of supremum on finite and infinite time intervals has been investigated in many papers (see
Bernyk et al. \cite{be:da:pe:08}, Bertoin \cite{be:96}, Bertoin et al. \cite{be:do:ma:08}, Furrer \cite{fu:98},
Harrison \cite{ha:77}, Hubalek and Kuznetsov \cite{hu:ku:11}, Kuznetsov \cite{ku:10}, Simon \cite{si:11}, Tak\'acs \cite{ta:65}, Zolotarev \cite{zo:64} and many others).
In Zolotarev \cite{zo:64} the Laplace transform of the distribution of the supremum on the infinite time interval for spectrally positive L\'evy processes is given.
In the recent works of Bernyk et al. \cite{be:da:pe:08}, Hubalek and Kuznetsov \cite{hu:ku:11} and Kuznetsov \cite{ku:10} a series representation for the density function of the supremum on finite intervals for stable L\'evy processes  and a certain class of L\'evy processes is given. 
In this article we determine the distribution of the supremum on finite and infinite intervals for spectrally positive L\'evy processes  
which are generalizations of  the pretty formulas of Tak\'acs \cite{ta:65} proven for L\'evy processes with finite variation.

Let $X=\{X(t): t\geq 0\}$ be a spectrally positive L\'evy process with characteristic function of the form
$$
\Exp\exp\{iuX(t)\}=\exp\left\{t\left[iau+\int_0^\infty(e^{iux}-1-iux\ind(x< 1))\,Q(dx) \right]\right\},
$$
where $a\in \mathbb{R}$ and $Q$ is a L\'evy measure (thus, $X$ is a L\'evy process without Gaussian component). We will investigate spectrally positive L\'evy processes $X$ with paths of infinite variation on finite intervals, which corresponds to the case that $\int_0^1 x\,Q(dx)=\infty$. Some of the results we will phrase in terms of the spectrally negative L\'{e}vy process $Y=-X$.
We assume throughout the article that $u>0$.

The results of Tac\'acs for L\'evy processes with non-negative increments are the following (originally derived in a slightly more general setting, that is, for processes with non-negative interchangeable increments;
see Tak\'acs \cite{ta:65}).
\begin{theorem}\label{tac}
If $X$ is a L\'evy process with non-negative increments and $c>0$ then
\begin{eqnarray}\label{tac1}
\lefteqn{\Prob(\sup_{t\leq T} (X(t)-ct)>u)}\label{tacf}\\
&=&\Prob(X(T)-cT>u)+\int_0^T\Exp\left(\frac{X(T-s)}{T-s}-c\right)^-
{\rm d_s}\Prob(X(s)-cs\leq u)\,,\nonumber
\end{eqnarray}
where $x^-=-\min\{x,0\}$ and ${\rm d_s}\Prob(X(s)-cs\leq u)=\Prob(u\leq X(s)-cs\leq u+{\rm d}s)$.

If $0\leq \Exp X(1)<c$ then
\begin{equation}\label{tac2}
\Prob(\sup_{t<\infty} (X(t)-ct)>u)=(c-\Exp X(1))\int_0^\infty{\rm d_s}\Prob(X(s)-cs\leq u)\,.
\end{equation}

If $\Exp X(1)\geq c$ then $\Prob(\sup_{t<\infty} (X(t)-ct)>u)=1$.
\end{theorem}
We provide a rigorous definition of the measure ${\rm d_s}\Prob(X(s)-cs\leq u)=\Prob(u\leq X(s)-cs\leq u+{\rm d}s)$.
\begin{definition}
For $v\in \Bbb R$ and a process $X$ the measure ${\rm d_s}\Prob(X(s)\leq v)$  on $[0,\infty)$ is defined by the following integral
$$
\int_a^b g(s)\, {\rm d_s}\Prob(X(s)\leq v)=\lim_{\delta({\cal P})\rightarrow 0}\sum_{k=1}^n g(s_{k-1})
(F(v+\Delta s_{k-1}, s_{k-1})-F(v,s_{k-1}))\,,
$$
where $g$ is a function defined on the interval $[a,b]$, $F(v, s)=\Prob(X(s)\leq v)$, ${\cal P}=\{s_k\}_{k=0}^n$ is a partition of the interval $[a,b]$,
$\Delta s_{k-1}=s_k-s_{k-1}$ and $\delta({\cal P})$ is the mesh of the partion ${\cal P}$.
\end{definition}

\begin{remark}\rm 
Note that if $X$ has one-dimensional distributions that are absolutely continuous with respect to Lebesgue measure then ${\rm d_s}\Prob(X(s)\leq u)=f(u,s)\,{\rm d}s$, where $f(u,s)$ is a density function of the random variable $X(s)$.
\end{remark}
Further, if $X(s)$ or $Y(s)$ has a density function, their density functions will be denoted by $f(v,s)$ and $f(-v,s)$, respectively.
Th.~\ref{tac} completely solves the problem of the supremum distribution for spectrally positive L\'evy processes with finite variation. The proof of Th.~\ref{tac} uses a generalization of the classical ballot theorem and an approximation argument (see Tak\'acs \cite{ta:65}).

Similarly in Tak\'acs \cite{ta:65} a formula for the process $ct-X(t)$ is derived (the process $ct-X(t)$ is a spectrally negative L\'evy process with finite variation).
\begin{theorem}\label{tacneg}
If $X$ is a L\'evy process with non-negative increments and $c>0$ then
\begin{eqnarray}
\Prob(\sup_{t\leq T} (ct-X(t))>u)&=&u\int_{u/c}^T
s^{-1}\,{\rm d_s}\Prob(X(s)-cs\leq -u)\nonumber\\
&=&u\int_{0}^T
s^{-1}\,{\rm d_s}\Prob(cs-X(s)\leq u)\,,\label{taca1}
\end{eqnarray}
where ${\rm d_s}\Prob(cs-X(s)\leq u)=\Prob(u\leq cs-X(s)\leq u+{\rm d}s)$.
\end{theorem}
\begin{remark}\rm
The formula (\ref{taca1}) is the well-known identity of Kendall which is valid for any spectrally negative L\'evy process
not equal to a subordinator (see Kendall \cite{ke:57} or Borovkov and Burq \cite{bo:bu:01} and references therein).
Here we have Kendall's identity for a spectrally negative L\'evy process $ct-X(t)$  with finite variation.
\end{remark}
Let us recall Kendal's identity which will be used in the proof of our main result Theorem~\ref{mi}.
Let
\begin{equation}\label{stt}
S(z)=\inf\{t\geq 0: Y(t)>z\}\,,
\end{equation}
where $z\geq 0$.
\begin{theorem}
For any spectrally negative L\'evy process $Y$ that is not a subordinator and $t, z>0$, we have
the identity
$$
t\,\Prob(S(z)\in {\rm d}t)=z\,\Prob(Y(t)\in {\rm d}z)\,{\rm d}t\,.
$$
\end{theorem}
\begin{remark}\rm
Under the condition that the one-dimensional distributions of $Y$
are absolutely continuous with density functions $f(-v,s)$, the density of the random variable $S(z)$ is
given by
$$
\frac{\Prob(S(z)\in {\rm d}t)}{{\rm d}t}=\frac{z}{t}\,f(-z, t)\,.
$$
\end{remark}

\section{The infinite variation case}
We extend the results of Tac\'acs~\cite{ta:65} to the case of L\'evy processes with infinite variation.
\begin{theorem}\label{mi}
If the one-dimensional distributions of $X$ are absolutely continuous, then
\begin{eqnarray}\label{mi1}
\lefteqn{\Prob(\sup_{t\leq T} X(t)>u)}\nonumber\\
&=&\Prob(X(T)>u)+\int_0^T\frac{\Exp(X(T-s))^-}{T-s}\,
f(u,s)\,{\rm d}s\,,\label{mainf}
\end{eqnarray}
where $x^-=-\min\{x,0\}$ and $f(u,s)$ is a density function of $X(s)$.
\end{theorem}
\begin{remark}\rm
Note that for spectrally positive L\'evy processes $\Exp (X(1))^-<\infty$ which gives that
$\Exp X(1)>-\infty$ (see e.g. Sato \cite{sa:99}, Theorem 26.8).
\end{remark}
\begin{proof}
The proof will be based on Kendall's identity, duality, strong Markov property and Hunt's switching identity. Let $P_t$ and $\widehat{P}_t$ be Markov semigroups of the processes $Y$ and $-Y=X$,
respectively, killed upon entering the negative half-line $(-\infty,0)$ (see e.g. Bertoin \cite{be:96}, Sections 0.1 and II.1). By $(\Prob_x,\, x\in \mathbb R)$ we denote the family of
measures conditioned on $\{Y(0)=x\}$ with $\Prob_0=\Prob$. Thus by Hunt's
switching identity (see e.g. Bertoin \cite{be:96}, Theorem II.1.5) it holds for
nonnegative measurable functions $f,g$ and every $t\geq 0$ that
\begin{equation}\label{eq:hunt}
\int_\R P_t f(x) g(x) \td x = \int_\R f(x) \widehat{P}_t g(x)\td x.
\end{equation}
Moreover the right hand side of (\ref{eq:hunt}) is as follows
\begin{eqnarray}
\lefteqn{\int_\R f(x) \widehat{P}_t g(x)\td x =
\int_\R f(x) \Exp_{-x}[g(-Y(t))\ind_{\{t< S(0)\}}]\td x}\nonumber\\
&=& \int_\R f(x)\Exp_{-x}[g(-Y(t))]\td x\label{shi}\\
 &&-\int_\R f(x)\td x\int_0^t \Exp_{-x}[g(-Y(t))| S(0)=s]\Prob_{-x}(S(0)\in \td s)\label{shiafter}\\
&=& \int_\R \Exp_{x}[f(Y(t))] g(x)\td x-\int_\R f(x)\td x\int_0^t \Exp[g(-Y(t-s))]\Prob(S(x)\in \td s)\nonumber
\end{eqnarray}
where in the last equality for the term (\ref{shi}) we use Bertoin \cite{be:96}, Proposition II.1.1,
and for the term  (\ref{shiafter}) we have $\Prob_{-x}(S(0)\in \td s)=\Prob(S(x)\in \td s)$ and
by strong Markov property and by the fact $Y(S(0))=0$ ($Y$ does not jump upwards) we get
$$
\Exp_{-x}[g(-Y(t))| S(0)=s]=\Exp[g(-Y(t-s))]
$$
for $s<t$. Taking $f$ disappearing on the negative half-line and substituting Kendall's identity into the last subtrahend we obtain
\begin{eqnarray*}
\lefteqn{\int_\R f(x)\td x\int_0^t \Exp[g(-Y(t-s))]\Prob(S(x)\in \td s)}\\
&=&\int_0^\infty f(x)\td x\int_0^t \Exp[g(-Y(t-s))]\,\frac{x}{s}\,f(-x,s)\td s\\
&=&\int_0^\infty f(x)\td x\int_\R g(z)\td z\int_0^t f(z,t-s)\,\frac{x}{s}\,f(-x,s)\td s\\
&=&\int_0^\infty f(x)\td x\int_\R g(z)\td z\int_0^t \frac{x}{t-s}\,f(z,s)f(-x,t-s)\td s\\
&=&\int_\R g(x)\td x\int_0^\infty f(z)\td z\int_0^t \frac{z}{t-s}\,f(x,s)f(-z,t-s)\td s\,,
\end{eqnarray*}
where in the last equality we swapped $x$ with $z$ and changed the order of integrals.
The left hand side of (\ref{eq:hunt}) is the following
$$
\int_\R P_t f(x) g(x) \td x
=\int_\R g(x)\td x\int_0^\infty f(z) \Prob_x(\inf_{s\leq t} Y(s)\geq 0, Y(t)\in \td z)\,.
$$
Now returning to (\ref{eq:hunt}) we get the following identity for measures
\begin{eqnarray*}
\lefteqn{\Prob_x(\inf_{s\leq t} Y(s)\geq 0, Y(t)\in \td z)\td x}\\
&=&
\Prob_x(Y(t)\in \td z)\td x-\td z \td x\int_0^t\frac{z}{t-s}f(x,s) f(-z, t-s)\td s\,,
\end{eqnarray*}
where $x, z>0$ which gives
\begin{equation}\label{more}
\Prob_x(\inf_{s\leq t} Y(s)< 0, Y(t)\in \td z)=
\int_0^t\frac{z}{t-s}\, \Prob(Y(t-s)\in \td z)f(x,s)\td s
\end{equation}
where $t, x ,z >0$. Thus we have
$$
\Prob(\sup_{s\leq t} X(s)>x,\, x-X(t)\in \td z)=
\int_0^t\frac{z}{t-s}\, \Prob(-X(t-s)\in \td z)f(x,s)\td s
$$
for $z>0$ and obviously
$$
\Prob(\sup_{s\leq t} X(s)>x, \,x-X(t)\in \td z)=
\Prob(x-X(t)\in \td z)
$$
for $z\leq0$. Integrating the last formula with respect to $z$
we get the thesis of the theorem.
\end{proof}

The formula (\ref{mi1}) was derived in Michna \cite{mi:11} for spectrally positve $\alpha$-stable L\'evy processes (see also Furrer \cite{fu:97} Prop. 2.7).

In fact by (\ref{more}) we proved a more general result which determines the joint distribution of 
$\inf_{t\leq T} Y(t)$ and $Y(T)$.
\begin{theorem}\label{tacneginf}
If $Y$ is a spectrally negative L\'evy process that is not a subordinator, then
$$
\Prob (\inf_{t\leq T} Y(t)< -x, \,Y(T)+x\in \td z)=
\int_0^T\frac{z}{T-s}\, \Prob(Y(T-s)\in \td z)p(-x,s)\td s,
$$
where $T, x ,z>0$ and here $p(x,s)$ is a density function of $Y(s)$.
\end{theorem}
\begin{remark}\rm 
It follows from the proof of Theorem \ref{mi} that the formulas of Theorems~\ref{mi} and \ref{tacneginf}
are valid for $X$ and $Y$ having also a Gaussian component.
\end{remark}
Note that the formula (\ref{mainf}) can also be obtained directly
from Tak\'acs formula (\ref{tac1}) using an approximation
argument. To outline the argument, we introduce
$$
N_\epsilon(t)=\sum_{s\leq t}\Delta X(s)\ind(\Delta X(s)\geq \epsilon)\,.
$$
The process $N_\epsilon$ is a compound Poisson process with positive jumps.
\begin{proposition}\label{weak}
\begin{equation}\label{sn2}
N_\epsilon(t)-(\int_{\epsilon}^1x\,Q(dx)-a)t\rightarrow X(t)\,,
\end{equation}
as $\epsilon\downarrow 0$ a.s. in the uniform topology.
\end{proposition}
\begin{proof}
The assertion follows from the proof of L\'evy-It\^o representation see e.g. Sato \cite{sa:99}.
\end{proof}
Substituting $X(t)=N_\epsilon(t)$ and $c=\int_{\epsilon}^1x\,Q(dx)-a$ to (\ref{tac1}) and letting $\epsilon$ tend to zero
we arrive at the formula (\ref{mainf}). To turn the sketched argument into a rigorous proof still
requires a justification to take the limits under the integral.

We also note that taking $T$ to infinity in (\ref{mainf}) yields the following formula (the passage to the limit $T\rightarrow\infty$ also needs a justification) for which we provide a proof under assumptions which are rather easy to check.
\begin{theorem}\label{forinfh}
Let $X$ be a spectrally positive L\'evy process such that $\Exp X(1)<0$
and let a function $g$ esist such that
\begin{equation}\label{addass}
\frac{\Exp(X(t))^-}{t}\leq g(t),\qquad t\in (0,1)
\end{equation} 
and
\begin{equation}\label{addass2}
\int_0^1 g(t)\,{\rm d} t<\infty\,.
\end{equation}
Moreover, we assume
\begin{equation}\label{addass3}
\sup_{t\geq 1}\frac{\Exp(X(t))^-}{t}<\infty
\end{equation}
and
\begin{equation}\label{addass31}
\lim_{t\rightarrow\infty}\frac{\Exp(X(t))^-}{t}=|\Exp X(1)|
\end{equation}
and for every $u>0$
\begin{equation}\label{addass4}
\lim_{s\rightarrow\infty}f(u,s)=0\,.
\end{equation}
Then
\begin{equation}\label{mi2}
\Prob(\sup_{t<\infty} X(t)>u)=|\Exp X(1)|\int_0^\infty f(u,s)\td s\,.
\end{equation}
\end{theorem}
\begin{proof}
By the formula (\ref{mi1}) and the assumption (\ref{addass31})
and Fatou lemma we get
\begin{equation}\label{calsk}
|\Exp X(1)|\int_0^{\infty}f(u,s)\,{\rm d}s\leq \Prob(\sup_{t< \infty} X(t)>u)\,.
\end{equation}
Using (\ref{mi1}) again we can write for $1\leq T_0<T-1$
\begin{eqnarray*}
\lefteqn{\Prob(\sup_{t\leq T} X(t)>u)}\nonumber\\
&=&\Prob(X(T)>u)+\int_{0}^{T_0}\frac{\Exp(X(T-s))^-}{T-s}\,
f(u,s)\,{\rm d}s \\
&&+ \int_{T_0}^{T-1}\frac{\Exp(X(T-s))^-}{T-s}\,
f(u,s)\,{\rm d}s  +\int_{T-1}^{T}\frac{\Exp(X(T-s))^-}{T-s}\,
f(u,s)\,{\rm d}s\,. 
\end{eqnarray*}
Thus if $T\rightarrow\infty$ we obtain
\begin{eqnarray}
\lefteqn{\Prob(\sup_{t<\infty} X(t)>u)}\nonumber\\
&\leq & |\Exp X(1)| \int_0^{T_0}
f(u,s)\,{\rm d}s\label{upper}\\
&&+
\left(\sup_{t\geq 1} \frac{\Exp(X(t))^-}{t}\right)\,\int_{T_0}^\infty f(u,s)\,{\rm d}s\nonumber
\\&&+
\lim_{T\to \infty}
\int_{T-1}^{T}\frac{\Exp(X(T-s))^-}{T-s}\,
f(u,s)\,{\rm d}s\,.\nonumber
\end{eqnarray}
Applying the assumptions  (\ref{addass}) and  (\ref{addass2})  and (\ref{addass4}) we obtain
$$\lim_{T\to \infty}
\int_{T-1}^{T}\frac{\Exp(X(T-s))^-}{T-s}\,
f(u,s)\,{\rm d}s
=0$$
because 
\begin{eqnarray*}
\int_{T-1}^{T}\frac{\Exp(X(T-s))^-}{T-s}\,
f(u,s)\,{\rm d}s&=& \int_0^1 \frac{\Exp(X(s))^-}{s}\,
f(u,T-s)\,{\rm d}s\\
&\leq&\int_0^1 g(s)\,
f(u,T-s)\,{\rm d}s\\
&\leq & \sup_{T-1<s<T} f(u,s)\int_0^1 g(s)\,{\rm d}s\,.
\end{eqnarray*}
Hence taking $T_0\rightarrow\infty$ in (\ref{upper}) we obtain
$$
\Prob(\sup_{t<\infty} X(t)>u)
\leq |\Exp X(1)| \int_0^{\infty}
f(u,s)\,{\rm d}s
$$
which together with (\ref{calsk}) completes the proof.
\end{proof}
In the case $\Exp X(1)\geq 0$ it is easy to show that $\Prob(\sup_{t<\infty} X(t)>u)=1$
(use the law of large numbers and Chung and Fuchs \cite{ch:fu:57} in the case $\Exp X(1)= 0$).

\begin{example}\rm
Let us consider the spectrally positive $\alpha$-stable L\'evy process $Z_\alpha$ with $1<\alpha\leq 2$ (that is the skewness parameter $\beta=1$ and the shift parameter $\mu=0$, see e.g. Samorodnitsky and Taqqu \cite{sa:ta:94}).
We will investigate the process $X(t)=Z_\alpha(t)-ct$ with $c>0$ which is a spectrally positive L\'evy process with $\Exp X(t)=-ct$. Note that
\begin{eqnarray}
\frac{\Exp(X(t))^-}{t}&=&\Exp\left(\frac{Z_\alpha(t)}{t}-c\right)^-\label{stable}\\
&=&\Exp\left(t^{1/\alpha-1}Z_\alpha(1)-c\right)^-\nonumber\\
&\leq& t^{1/\alpha-1} \Exp\left(Z_\alpha(1)-c\right)^-\nonumber\\
&=& g(t)\nonumber
\end{eqnarray}
where in the second line we used the self-similarity of $Z_\alpha$ and the last inequality is valid for $0<t\leq 1$ providing the function $g$ which satisfies the assumptions (\ref{addass}) and (\ref{addass2}). Since
\begin{equation}\label{stable1}
\left(t^{1/\alpha-1}Z_\alpha(1)-c\right)^-\leq |Z_\alpha(1)|+c
\end{equation}
for $t\geq 1$ and the right hand side of the last inequality is integrable we get
$$
\lim_{t\rightarrow\infty}\frac{\Exp(X(t))^-}{t}=\lim_{t\rightarrow\infty}\Exp\left(t^{1/\alpha-1}Z_\alpha(1)-c\right)^-=|\Exp X(1)|
$$
by Lebesgue dominated convergence theorem. By (\ref{stable}) and (\ref{stable1}) the assumption (\ref{addass3}) is satisfied. Moreover if $f(x)$ is the density function
of $Z_\alpha(1)$ then the density function of $Z_\alpha(s)-cs$ is 
$$
f(u,s)=s^{-1/\alpha}f(s^{-1/\alpha}(u+cs))
$$
and it is clear that $\lim_{s\rightarrow\infty}s^{-1/\alpha}f(s^{-1/\alpha}(u+cs))=0$. 
Thus using  (\ref{mi2}) of Th.~\ref{forinfh} we get
\begin{equation}\label{mi2st}
\Prob(\sup_{t<\infty} (Z_\alpha(t)-ct)>u)=c\int_0^\infty s^{-1/\alpha} f(s^{-1/\alpha}(u+cs))\,{\rm d}s\,.
\end{equation}
Applying a certain form of the density $f(x)$ for $1<\alpha<2$ for the parametrization as in Samorodnitsky and Taqqu \cite{sa:ta:94} (see e.g. Nolan \cite{no:97} and references therein) and the scale parameter $\sigma=1$ (that is
$Z_\alpha(1)$ has the scale parameter $\sigma=1$) we obtain
\begin{eqnarray*}
\lefteqn{\Prob(\sup_{t<\infty} (Z_\alpha(t)-ct)>u)}\\
&=&\frac{c}{\pi}\int_0^\infty{\rm d}s\,s^{-1/\alpha}
\int_0^\infty e^{-t^\alpha}\cos\left(ts^{-1/\alpha}(u+cs)-t^\alpha\tan{\frac{\pi\alpha}{2}}\right){\rm d}t\\
&=&\sum_{n=0}^\infty \frac{(-a)^n}{\Gamma(1+(\alpha-1)n)}\,u^{(\alpha-1)n}\,,
\end{eqnarray*}
where $a=c\cos(\pi(\alpha-2)/2)$ and
the last equality follows by comparing with the result of Furrer \cite{fu:98}  (the last expression is the Mittag-Leffler function).

For the standard Wiener process $W(t)$ similarly as above we get the following identity
$$
\Prob(\sup_{t<\infty} (W(t)-ct)>u)
=\frac{c}{\sqrt{2\pi}}\int_0^\infty s^{-1/2}\exp\left(-\frac{(u+cs)^2}{2s}\right){\rm d}s=\exp(-2uc)\,,
$$
where the last equality is the well-known result for the supremum distribution of Wiener process over the infinite time horizon (see e.g. Asmussen and Albrecher \cite{as:al:10}).

Similarly one can consider the distribution of the supremum on finite intervals. By the formula (\ref{mi1}) of Th.~\ref{mi} we derive (for simplicity we put $c=0$, for $c\neq 0$ the formula will be a little more complicated)
\begin{eqnarray*}
\lefteqn{\Prob(\sup_{t\leq T} Z_\alpha(t)>u)}\\
&=&\Prob(Z_\alpha(T)>u)\\
&&+\frac{\Exp (Z(1))^-}{\pi}\int_0^T{\rm d}s\,(T-s)^{1/\alpha-1}\,s^{-1/\alpha}
\int_0^\infty e^{-t^\alpha}\cos\left(ts^{-1/\alpha}u-t^\alpha\tan{\frac{\pi\alpha}{2}}\right){\rm d}t\,,
\end{eqnarray*}
where
$$
\Prob(Z_\alpha(T)>u)=\frac{T^{-1/\alpha}}{\pi}\int_u^{\infty}{\rm d}x\int_0^\infty e^{-t^\alpha}\cos\left(tT^{-1/\alpha}x-t^\alpha\tan{\frac{\pi\alpha}{2}}\right){\rm d}t
$$
and
$$
\Exp(Z(1))^-=\frac{1}{\pi}\int_{-\infty}^{0}{\rm d}x\,x\int_0^\infty e^{-t^\alpha}\cos\left(tx-t^\alpha\tan{\frac{\pi\alpha}{2}}\right){\rm d}t\,;
$$
compare the formula with Bernyk et al. \cite{be:da:pe:08} where they get a series representation for the density function of the supremum distribution, see also Bertoin et al. \cite{be:do:ma:08} and
Hubalek and Kuznetsov \cite{hu:ku:11}.
\end{example}

\begin{example}\rm
Assume that $X(t)$ is a compound Poisson process with negative drift $ct$ and nonnegative jumps perturbed by a spectrally positive $\alpha$-stable L\'evy process $Z_\alpha(t)$ with $1<\alpha\leq 2$ .
Then
$$(X(t))^-\leq  (Z_\alpha(t)-ct)^-$$
so using (\ref{stable}) and (\ref{stable1})
one can easily check that 
the assumptions (\ref{addass}), (\ref{addass2}), (\ref{addass3}) and (\ref{addass31}) are satisfied.
Under mild conditions on the distribution of the compound Poisson process we can check the assumption (\ref{addass4}).
\end{example}

In some cases the supremum distribution can be identified by using just the strong Markov property.
Indeed, let us consider the spectrally negative $\alpha$-stable L\'evy process $Z_\alpha$ with $1<\alpha\leq 2$ without any drift
(that is the skewness parameter $\beta=-1$ and the shift parameter $\mu=0$).  Thus, let $\tau=\inf\{t>0: Z_\alpha(t)>u\}$ where $u\geq 0$ then $\{\sup_{t\leq T} Z_\alpha (t)>u\}=\{\tau<T\}$ a.s. Since the process $Z_\alpha$ is spectrally negative
(it has no positive jumps), we have $Z_\alpha(\tau)=u$. By the strong Markov property $Z^*_\alpha(t)=Z_\alpha(t+\tau)-Z(\tau)$ is a L\'evy process with the same distribution as $Z_\alpha$
and independent of $\tau$. We know that $\Prob(Z_\alpha(s)>0)=1/\alpha$. Thus, for $u\geq 0$ we have:
\begin{eqnarray*}
\lefteqn{\Prob(\sup_{t\leq T} Z_\alpha (t)>u)}\\
&=&\Prob(\sup_{t\leq T} Z_\alpha (t)>u,\, Z_\alpha(T)>u)+
\Prob(\sup_{t\leq T} Z_\alpha (t)>u,\, Z_\alpha(T)\leq u)\\
&=&\Prob(Z_\alpha(T)>u)+\Prob(\sup_{t\leq T} Z_\alpha (t)>u,\, Z^*_\alpha(T-\tau)\leq 0)\\
&=&\Prob(Z_\alpha(T)>u)+\Prob(\tau<T,\, Z^*_\alpha(T-\tau)\leq 0)\\
&=&\Prob(Z_\alpha(T)>u)+\int_0^T\Prob(Z^*_\alpha(T-s)\leq 0)\,{\rm d_s}\Prob(\tau<s)\\
&=&\Prob(Z_\alpha(T)>u)+\left(1-\frac{1}{\alpha}\right)\Prob(\tau<T)\\
&=&\Prob(Z_\alpha(T)>u)+\left(1-\frac{1}{\alpha}\right)\Prob(\sup_{t\leq T} Z_\alpha (t)>u),
\end{eqnarray*}
which gives
$$
\Prob(\sup_{t\leq T} Z_\alpha (t)>u)=\alpha\Prob(Z_\alpha (T)>u)\,;
$$
compare the last formula with the result of Michna \cite{mi:13}.

In this paper we studied the distribution of suprema for L\'{e}vy processes with jumps of single sign.
The general case of L\'evy processes with jumps of either sign
(for example, symmetric L\'evy processes) is much more complicated (see, for example, Kwa\'snicki et al.~\cite{kw:ma:ry:13}).

\section*{Acknowledgements}
This work is partially supported by the Ministry of Science and
Higher Education of Poland under the grant DEC-2013/09/B/ST1/01778
(2014-2016). The second author also kindly acknowledges partial support by the project RARE -318984, a Marie Curie IRSES Fellowship within the 7th European
Community Framework Programme.\\ 

\bibliographystyle{plainnat}

\begin{thebibliography}{99}

\bibitem{as:al:10}
{\sc Asmussen, S. and Albrecher, H.} (2010).
{\em Ruin Probabilities}. 2nd~edn. World Scientific Publishing, Singapore.

\bibitem{be:da:pe:08}
{\sc Bernyk, V.,  Dalang, R.C. and Peskir, G.} (2008).
The Law of the Supremum of a Stable L\'evy Process with No Negative Jumps.
{\em Annals of Probability} {\bf 36,} pp. 1777--1789.

\bibitem{be:96}
{\sc Bertoin, J.} (1996).
{L\'evy processes}. Cambridge University Press, Cambridge.

\bibitem{be:do:ma:08}
{\sc Bertoin, J.,  Doney, R.A. and  Maller, R.A.} (2008).
Passage of L\'evy processes across power law boundaries at small times.
{\em Annals of Probability} {\bf 36,} pp. 160--197.


\bibitem{bo:bu:01}
{\sc Borovkov, K. and Burq, Z.} (2001). Kendall's identity for the first crossing time revisited. {\em Electron. Comm. Probab.} {\bf 6,} pp. 91--94.


\bibitem{ch:fu:57}
{\sc Chung, K.L. and Fuchs, W.H.J.} (1951). On the distribution of values of sums of random variables. {\em Mem. Amer. Math. Soc.} {\bf 6} (1951), pp. 1--12.


\bibitem{fu:97}
{\sc Furrer, H.} (1997). Risk theory and heavy-tailed L\'evy processes.
Ph. D. Thesis ETH Z\"urich.

\bibitem{fu:98}
{\sc Furrer, H.} (1998). Risk processes perturbed by $\alpha$-stable L\'evy motion. {\em Scand. Actuar. J.} {\bf 10}, pp. 23--35.

\bibitem{ha:77}
{\sc Harrison, J. M.} (1977). The supremum distribution of a L\'evy process with
no negative jumps. {\em Advances in Applied Probability} {\bf 9} pp. 417--422.

\bibitem{hu:ku:11}
{\sc Hubalek, F. and Kuznetsov, A.} (2011).
A convergent series representation for the density of the supremum of a stable process.
{\em Electronic Communications in Probability} {\bf 16}, pp. 86--95.

\bibitem{ke:57}
{\sc Kendall, D.G.} (1957). Some problems in the theory of dams. \emph{J. Royal Stat. Soc. Ser. B} \textbf{19}, pp. 207--212.

\bibitem{ku:10}
{\sc Kuznetsov, A.} (2010).
Wiener-Hopf factorization and distribution of extrema for a family of L\'evy processes.
{\em Annals of Applied Probability} {\bf 20}, pp. 1801--1830.


\bibitem{kw:ma:ry:13}
{\sc Kwa\'snicki, M., Ma{\l}ecki, J. and Ryznar, M.} (2013).
Suprema of L\'evy processes. {\em Annals of Probability} {\bf 41,} pp. 247--265.


\bibitem{mi:11}
{\sc Michna, Z.} (2011).
Formula for the supremum distribution of a spectrally positive $\alpha$-stable L\'evy process.
{\em Statistics and Probability Letters} {\bf 81,} pp. 231--235.

\bibitem{mi:13}
{\sc Michna, Z.} (2013).
Explicit formula for the supremum distribution of a spectrally negative stable process. {\em Electron. Commun. Probab.} {\bf 18}, No. 10, pp. 1--6.

\bibitem{no:97}
{\sc Nolan, J.P.} (1997). Numerical calculation of stable densities and distribution functions. {\em Stochastic Models} {\bf 13}, pp. 759--774.

\bibitem{sa:ta:94}
{\sc Samorodnitsky, G. and Taqqu, M.} (1994).
{\em Non-Gaussian Stable Processes: Stochastic Models with
Infinite Variance}. Chapman and Hall, London.

\bibitem{sa:99}
{\sc Sato, K.} (1999). {\em L\'evy processes and infinitely divisible distributions}.
Cambridge University Press, Cambridge.

\bibitem{si:11}
{\sc Simon, T.} (2011). 
Hitting densities for spectrally positive stable processes.
{\em Stochastics: An International Journal of Probability and Stochastic Pro-
cesses} {\bf 83}(2), pp. 203--214.

\bibitem{ta:65}
{\sc Tak\'acs, L.} (1965).
On the distribution of the supremum for stochastic processes with interchangeable increments.
{\em Transactions of the American Mathematical Society} {\bf  119}, pp. 367--379.

\bibitem{zo:64}
{\sc Zolotarev, V. M.} (1964). 
The first passage time of a level and the behavior
at infinity for a class of processes with independent increments. {\em Theory of
Probability and Its Applications} {\bf9}(4), pp. 653--662.



\end{thebibliography}

\end{document}